\documentclass[a4paper]{article}
\usepackage[margin=3cm]{geometry}
\usepackage{authblk}

\usepackage{amsthm,amssymb,amsmath}
\usepackage{lineno}
\usepackage{natbib}
\usepackage{graphics}
\usepackage{xcolor}
\usepackage{xfrac}
\usepackage{booktabs}

\newcommand{\bx}{\boldsymbol{x}}

\newcommand{\bh}{\boldsymbol{h}}

\newcommand{\bW}{\boldsymbol{W}}

\newcommand{\bg}{\boldsymbol{g}}

\newcommand{\subW}{subW} 

\newcommand{\edr}{\mathrm{e}}

\newtheorem{theorem}{Theorem}[section]
\newtheorem{proposition}{Proposition}[section]

\newtheorem{corollary}{Corollary}[section]
\newtheorem{definition}{Definition}[section]
\theoremstyle{remark}

\usepackage{hyperref}
\hypersetup{
	colorlinks=true,
	urlcolor=blue,
	citecolor=blue,
	linktoc=all,
	linkcolor=blue}

\begin{document}

\title{Sub-Weibull distributions: \\
generalizing sub-Gaussian and sub-Exponential properties to heavier-tailed distributions\footnote{
M. Vladimirova and J. Arbel are funded by Grenoble Alpes Data Institute, supported by the French National Research Agency under the  ``Investissements d'avenir'' program (ANR-15-IDEX-02). 
H. Nguyen is funded by the Australian Research Council grants: DE170101134 and DP180101192. 
\par Corresponding author: mariia.vladimirova@inria.fr}
}

\author[1]{Mariia Vladimirova}
\author[1]{St\'ephane Girard}
\author[2]{Hien Nguyen}
\author[1]{Julyan Arbel}
\affil[1]{Univ. Grenoble Alpes, Inria, CNRS, LJK, 38000 Grenoble, France}
\affil[2]{Department of Mathematics and Statistics, La Trobe University, Victoria 3086, Australia}

\date{}

\maketitle

\begin{abstract}
We propose the notion of \textit{sub-Weibull} distributions, which are characterised by tails lighter  than (or equally light as) the right tail of a  Weibull distribution. This novel class generalises the sub-Gaussian and sub-Exponential families to potentially heavier-tailed distributions. Sub-Weibull distributions are parameterized by a positive tail index $\theta$ and reduce to sub-Gaussian distributions for $\theta=1/2$ and to sub-Exponential distributions for $\theta=1$. A characterisation of the sub-Weibull property based on moments and on the moment generating function is provided and properties of the class are studied. An estimation procedure for the tail parameter is proposed and is applied to an example stemming from Bayesian deep learning.
\end{abstract}

\section{Introduction and definition}

Sub-Gaussian distributions, respectively sub-Exponential, are characterized by their tails being upper bounded by Gaussian, respectively Exponential, tails. 
More precisely, we say that a random variable $X$ is sub-Gaussian, resp. sub-Exponential, if there exist positive constants $a$ and $b$ such that 
\begin{align}\label{eq:sub-G-and-E}
	\mathbb{P}(|X| \ge x) \le a \exp(-bx^2),\,\,\,\text{resp.}\,\,\,
	\mathbb{P}(|X| \ge x) \le a \exp(-bx),\,\,\, \text{for all}\,x>0.
\end{align}
These properties have been intensely studied in the recent years due to their relationship with  various fields of probability and statistics, including concentration, transportation and PAC-Bayes  inequalities~\citep{boucheron2013concentration,raginsky2013concentration,van2014probability,catoni2007pac}, the missing mass problem \citep{ben2017concentration}, bandit problems \citep{bubeck2012regret} and singular values of random matrices \citep{rudelson2010non}. 

It is tempting to generalise~\eqref{eq:sub-G-and-E} by considering the class of distributions satisfying
\begin{align}\label{eq:sub-W-tail-def}
	\mathbb{P}(|X| \ge x) \le a \exp\left(-bx^{1/\theta}\right),\,\,\, \text{for all }x>0,\,\,\, \text{for some } \theta,a,b>0.
\end{align}
This is the goal of the present note. Since a \textit{Weibull} random variable $X$ on $\mathbb{R}_+$ is defined by a survival function, for $x>0$, 
\begin{align}\label{eq:Weibull-survival}
	\bar{F}(x) = \mathbb{P}(X \ge x) = \exp\left(-bx^{1/\theta}\right),\,\,\, \text{for some}\,\,\,b>0,\theta>0,
\end{align}
we term a distribution satisfying~\eqref{eq:sub-W-tail-def} a \textit{sub-Weibull} distribution \citep[see][for a detailed account on the Weibull distribution]{rinne2008weibull}.
\begin{definition}[Sub-Weibull random variable]
\label{def:subweibull}
   A random variable $X$, satisfying~\eqref{eq:sub-W-tail-def} for some positive $a$, $b$ and $\theta$, is called a sub-Weibull random variable with tail parameter $\theta$, which is denoted by $X \sim \subW(\theta)$. 
\end{definition}

Interest in such heavier-tailed distributions than Gaussian or Exponential arises in our experience from their emergence in the field of Bayesian deep learning  \citep{vladimirova2018bayesian}. While writing this note, we were made aware of the preprint: \cite{kuchibhotla2018moving}, which, independent of our work, also introduces sub-Weibull distributions but from a different perspective. The definition proposed by \cite{kuchibhotla2018moving} is based on Orlicz norm  \citep[building upon][]{wellner2017bennett} and is equivalent to Definition~\ref{def:subweibull}. While \cite{kuchibhotla2018moving} focus on establishing tail bounds and rates of convergence for problems in high dimensional statistics, including covariance estimation and linear regression, under the sole sub-Weibull assumption, we are focused on proving sub-Weibull characterization properties. In addition, we illustrate their link with deep neural networks, not in the form of a model assumption as in \cite{kuchibhotla2018moving}, but as a characterisation of the prior distribution of deep neural networks units. We further note that \cite{Hao:2019aa} has also applied the notion of sub-Weibull distributions to the problem of constructing confidence bounds for bootstrapped estimators and cite \cite{kuchibhotla2018moving} and \cite{vladimirova2018bayesian} as sources.

The outline of the paper is as follows. We define sub-Weibull distributions, and prove characteristic properties in Section~\ref{sec:sub-weibull}, while some concentration properties are presented in Section~\ref{sec:concentration-prop}. Finally,  Section~\ref{sec:experiments} provides an example of sub-Weibull distributions arising from Bayesian deep neural networks, as well as an estimation procedure for the tail parameter.

\section{Sub-Weibull distributions: characteristic properties}
\label{sec:sub-weibull}

Let $X$ be a random variable. When the $k$-th moment of $X$ exists, $k \ge 1$, we denote $\| X \|_k = \left( \mathbb{E}[ |X|^k]\right)^{1/k}$. The following theorem states different equivalent distribution properties, such as tail decay, growth of moments, and inequalities involving the moment generating function (MGF) of $|X|^{1/\theta}$. The proof of this result shows how to transform one type of information about the random variable into another. Our characterizations are inspired by the characterizations of sub-Gaussian and sub-Exponential distributions in  \cite{vershynin2018high}.

\begin{theorem}[Sub-Weibull equivalent properties]
  Let $X$ be a random variable. Then the following properties are equivalent: 
  \begin{enumerate}
    \item\label{item:1} The tails of $X$ satisfy 
    \begin{align*}
      \exists K_1>0 \quad\text{such that}\quad
      \mathbb{P}(|X| \ge x) \le 2\exp\left( - (x/ K_1)^{1/\theta} \right) \quad \text{for all } x \ge 0.
    \end{align*}
    \item\label{item:2} The moments of $X$ satisfy 
     \begin{align*}
      \exists K_2>0 \quad\text{such that}\quad
      \| X \|_k  \le K_2 k^{\theta} \quad \text{for all } k \ge 1. 
    \end{align*}
    \item\label{item:3} The MGF of $|X|^{1/\theta}$ satisfies
     \begin{align*}
      \exists K_3>0 \quad\text{such that}\quad
      \mathbb{E} \left[ \exp \left( (\lambda |X|)^{1/\theta} \right) \right] \le \exp\left((\lambda K_3)^{1/\theta} \right)
    \end{align*}
    for all $\lambda$ such that  $0 < \lambda \le  1/K_3$.
    
    \item\label{item:4} The MGF of $|X|^{1/\theta}$ is bounded at some point, namely
     \begin{align*}
      \exists K_4>0 \quad\text{such that}\quad
      \mathbb{E} \left[ \exp \left( (|X| / K_4)^{1/\theta} \right)\right] \le 2.
    \end{align*}
  \end{enumerate}
\label{th:subWeibull}
\end{theorem}

An illustration of symmetric  sub-Weibull distributions is represented in Figure~\ref{pic:layers}  for different values of the tail parameter $\theta$. Since only the tail is relevant for the illustration, the sub-Weibull survival functions $S_\theta$ depicted here are obtained from the survival functions $S^W_\theta$  of Weibull random variables with shape parameter $1/\theta$, scale parameter 1. More specifically, only the part of $S^W_\theta$ to the right of the $95$-th quantile is considered, and it is symmetrized around zero:
\begin{align*}
    S_\theta(x) = 
    10\edr^{-(x+\log(20)^\theta)^{1/\theta}} \quad \text{if } x\geq 0,\quad \text{and} \quad
    S_\theta(x) =    
    1-10\edr^{-(-x+\log(20)^\theta)^{1/\theta}} \quad \text{if } x< 0.
\end{align*}
\begin{figure}
\centering
  \includegraphics[width=.6\textwidth]{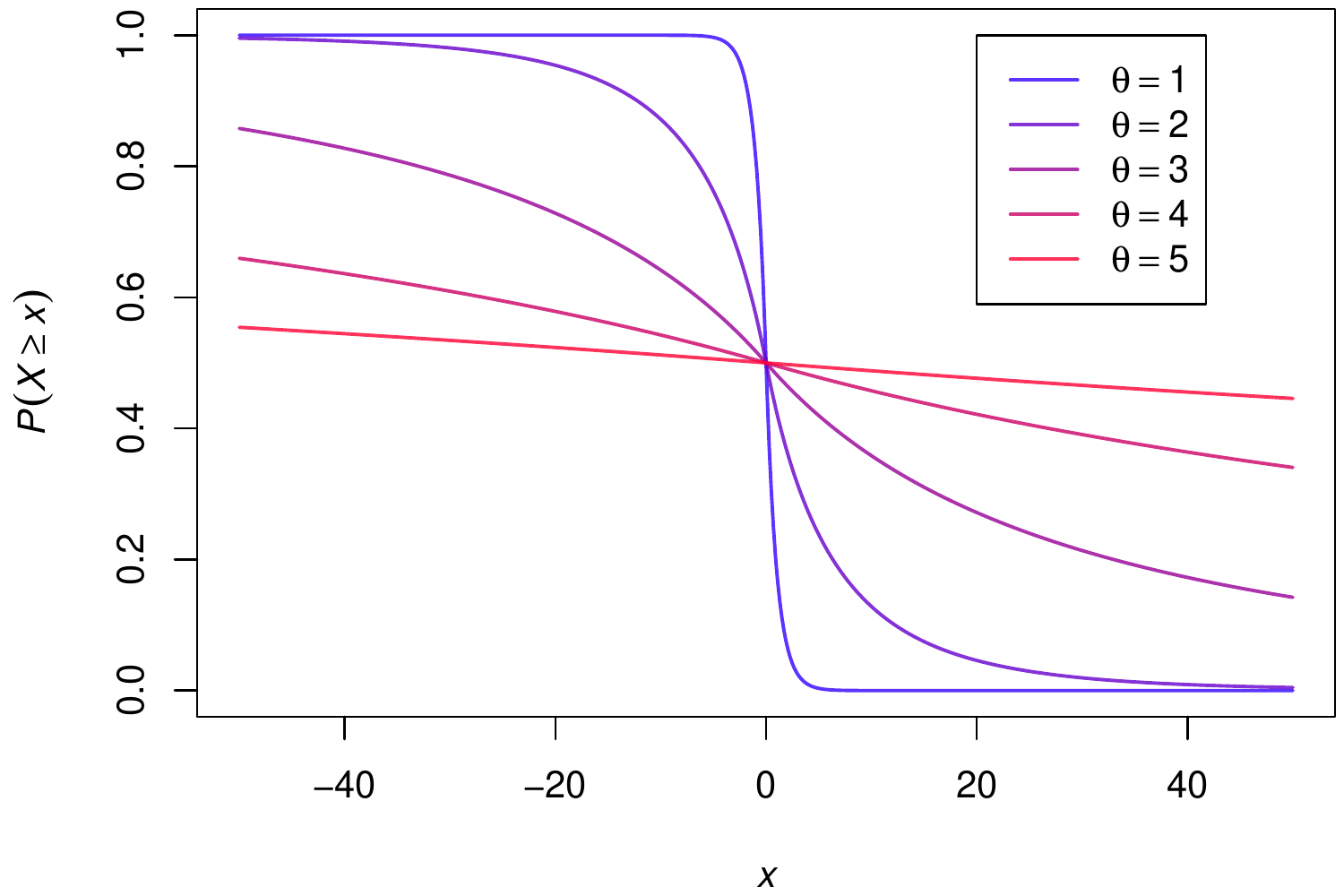}
   \caption{Illustration of sub-Weibull survival curves on  $\mathbb{R}$ with varying tail parameters $\theta$. }
 
   \label{pic:layers} 
\end{figure}
\begin{proof}
In the proof we use the notation $\lesssim$ between two positive sequences $(a_k)_k$ and $(b_k)_k$,  writing $a_k \lesssim b_k$, if there exists a constant $C > 0$ such that for all integer $k$, $a_k \le C b_k$.

  {\bf \ref{item:1} $\Rightarrow$ \ref{item:2}.} Assume Property~\ref{item:1} holds. 
    We have, for any $k\geq 1$,
    \begin{align*}
      \mathbb{E} \left[ |X|^k \right] &\overset{\text{(a)}}{=} \int_0^\infty \mathbb{P} \left(|X|^k > x \right) \mbox{d}x 
      = \int_0^\infty \mathbb{P} \left(|X| > x^{1/k}\right) \mbox{d}x \\
      &
       \overset{\text{(b)}}{\le} \int_0^\infty 2 \exp \left( - (x/K_1) ^{1/(k\theta)}  \right) \mbox{d}x =
      2 K_1 k \theta \int_0^\infty \mbox{e}^{-u} u^{k \theta-1} \mbox{d}u = 2 K_1 k \theta \,\Gamma \left( k \theta  \right)
      = 2 K_1  \,\Gamma \left( k \theta  +1 \right)
      \\
      & 
      \overset{\text{(c)}}{\lesssim} 2 K_1 \left( k \theta + 1 \right)^{k \theta + 1} =  2 K_1 \left( k \theta + 1 \right) \left( k \theta + 1 \right)^{k \theta},
    \end{align*}
    where (a) is by the so-called  integral identity for $|X|^k$, (b) is by Property~\ref{item:1}, and (c) comes from Stirling's formula, yielding $\Gamma(u) \lesssim u^u$.   
Taking the $k$-th root of the expression above yields
    \begin{align*}
      \|X\|_k \lesssim (2 K_1)^{1/k} (k \theta + 1)^{1 / k} \left( k \theta + 1 \right)^{\theta} 
      \lesssim  k^\theta,
    \end{align*}
which is to Property~\ref{item:2}.

  {\bf \ref{item:2} $\Rightarrow$ \ref{item:3}.} Assume Property~\ref{item:2} holds. Recalling the Taylor series expansion of the exponential function, we obtain
  \begin{align*}
    \mathbb{E} \left[ \exp \left( \lambda^{1/\theta} |X|^{1/\theta} \right) \right] &=  \mathbb{E} \left[ 1 + \sum_{k=1}^\infty \frac{(\lambda^{1/ \theta} |X|^{1/\theta})^k}{k!}\right] 
    = 1 + \sum_{k=1}^\infty \frac{\lambda^{k/\theta} \mathbb{E}[|X|^{k/\theta}]}{k!}.
  \end{align*}
  Property~\ref{item:2} guarantees that  $\mathbb{E}[|X|^{k/\theta}] \le K_2^{k/\theta} (k/\theta)^{k}$ for some  $K_2$
  and for any $k \ge \theta$. There exist only a finite number of integers $k$ that are less than $\theta$, thus there exists a constant $\tilde K_2$ such that $\mathbb{E}[|X|^{k/\theta}] \le \tilde K_2^{k/\theta} (k/\theta)^{k}$ for integers  $k \ge \theta$. 
Defining $K_2'=\max\{K_2,\tilde K_2\}$
 implies $\mathbb{E}[|X|^{k/\theta}] \le K_2'^{k/\theta} (k/\theta)^{k}$ for all integers $k \ge 1$.
 By Stirling's approximation we have $k! \gtrsim (k/\mbox{e})^k$. Substituting these two bounds, we get 
  \begin{align*}
    \mathbb{E} \left[ \exp \left( \lambda^{1/\theta} |X|^{1/\theta} \right) \right] 
    \lesssim 1 + \sum_{k=1}^\infty \frac{\lambda^{k/\theta} K_2'^{k/\theta} (k/\theta)^k}{(k / \mbox{e})^k} 
    = \sum_{k=0}^\infty  K_2'^{k/\theta} (\mbox{e} \lambda^{1/\theta}/\theta)^k 
    = 
 \frac 1 {1 - \mbox{e} (K_2' \lambda)^{1/\theta}/\theta}
  \end{align*}
  provided that  $\mbox{e} (K_2' \lambda)^{1/\theta}/\theta < 1$, in which case the geometric series above converges. To bound this quantity further, we can use the  inequality $\frac 1 {1 - x} \le \mbox{e}^{2x}$, which is valid for $x \in [0, 1/2]$. It follows that
     \begin{align*}
    \mathbb{E} \left[ \exp \left( \lambda^{1/\theta} X^{1/\theta} \right) \right] 
    \le \exp \left( 2\mbox{e} (K_2' \lambda)^{1/\theta}/\theta \right) 
   \end{align*}
  for all $\lambda$ satisfying $0 < \lambda \le \frac 1 {K_2'} \left(\frac {\theta} {2 \mbox{e}} \right)^{\theta}$ and some positive $K_2'$ defined above. This yields Property~\ref{item:3} with $K_3 = K_2' (2 \mbox{e} /\theta)^{\theta}$.

  {\bf \ref{item:3} $\Rightarrow$ \ref{item:4}.} Assume Property~\ref{item:3} holds. Take $\lambda = 1 / K_4$, where $K_4 \ge K_3 / (\ln 2)^{\theta}$. This yields Property~\ref{item:4}.

  {\bf \ref{item:4} $\Rightarrow$ \ref{item:1}.} Assume Property~\ref{item:4} holds. Then, by Markov's~inequality and Property~\ref{item:3}, we obtain 
  \begin{align*}
    \mathbb{P} \bigl( |X| > x\bigr) = \mathbb{P} \left( \mbox{e}^{(|X|/K_4)^{1/\theta}} > \mbox{e}^{(x/K_4)^{1/\theta}} \right) 
    \le \frac{\mathbb{E} [ \mbox{e}^{(|X|/K_4)^{1/\theta}}]}{\mbox{e}^{(x/K_4)^{1/\theta}}} \le 2 \mbox{e}^{- (x/K_4)^{1/\theta}}.
  \end{align*}
  This proves Property~\ref{item:1} with $K_1 = K_4$. 
\end{proof}
Informally, the tails of a $\subW(\theta)$ distribution are dominated by (i.e. decay at least as fast as) the tails of a Weibull variable with shape parameter\footnote{Weibull distributions are commonly parameterized by a shape parameter $\kappa$. Here we use instead $\theta=1/\kappa$ for the convenience that the larger the tail parameter $\theta$, the heavier the tails of the sub-Weibull distribution.} equal to  $1/\theta$. 
Sub-Gaussian and sub-Exponential variables, which are commonly used, are special cases of sub-Weibull random variables with tail parameter $\theta = 1/2$ and $\theta = 1$, respectively, see Table~\ref{table:distributions}. 
%
\renewcommand{\arraystretch}{1.3}
\begin{table}[!ht]
\centering
\caption{Sub-Gaussian, sub-Exponential and sub-Weibull distributions comparison in terms of tail $P( |X| \ge x)$ and moment condition, with $K_1$ and $K_2$ some positive constants. The first two are a special case of the last with $\theta=1/2$ and $\theta=1$, respectively.}
\begin{tabular}{@{}lll@{}}
\toprule
Distribution                         & Tails          & Moments          \\ \toprule
Sub-Gaussian                 & $\mathbb{P}(|X| \ge x) \le 2\edr^{ - (x / K_1)^2}$                   & $\|X\|_k \le K_2 \sqrt{k}$ \\ \hline
Sub-Exponential              & $\mathbb{P}(|X| \ge x) \le 2\edr^{ - x / K_1 }$                  & $\|X\|_k \le K_2 k$ \\  \hline
Sub-Weibull       & $\mathbb{P}(|X| \ge x) \le 2\edr^{ - (x / K_1)^{1/\theta}}$             & $\|X\|_k \le K_2 k^{\theta}$ \\ \bottomrule
\end{tabular}

\label{table:distributions}
\end{table}
Sub-Gaussian distributions are sub-Exponential as well. Such inclusion properties are generalized to sub-Weibull distributions in the following proposition.  
\begin{proposition}[Inclusion]
\label{prop:inclusion}
Let $\theta_1$ and $\theta_2$ such that $0 < \theta_1 \leq \theta_2$ be two sub-Weibull tail parameters. The following inclusion holds 
\begin{align*}
	\subW(\theta_1) \subset \subW(\theta_2). 
\end{align*}
\end{proposition}
\begin{proof}
	For $X \sim \subW(\theta_1)$, there exists some constant $K_2 > 0$ such that for all $k > 0$, $\|X\|_k \le K_2 k^{\theta_1}$. Since $k^{\theta_1} \le k^{\theta_2}$ for all $k \ge 1$, this yields $\| X\|_k \le K_2 k^{\theta_2}$, which by definition implies $X \sim \subW(\theta_2)$. 	
\end{proof}
Let a random variable $X$ follow a sub-Weibull distribution with tail parameter $\theta$. 
Due to the property of inclusion from Proposition~\ref{prop:inclusion}, the sub-Weibull definition provides an upper bound for the tail. 
In order to address the question of a lower bound on the tail parameter of some sub-Weibull distribution, we rely on an optimal tail parameter for that distribution through the use of the moment Property~\ref{item:2} of Theorem~\ref{th:subWeibull}.
To this aim, we introduce the notation of asymptotic equivalence between two positive sequences $(a_k)_k$ and $(b_k)_k$ as by $a_k \asymp b_k$, defined by the existence of constants $D \ge d > 0$ such that 
\begin{align}
  \label{asymptotic_equivalence}
   d \le \frac{a_k}{b_k} \le D, \quad \text{for all } k \in \mathbb{N}. 
\end{align}
We can now introduce the notion of optimal sub-Weibull tail coefficient, and provide a moment-based condition for optimality to hold.
\begin{proposition}[Optimal sub-Weibull tail coefficient and moment condition]\label{prop:optimal_moment_condition}
Let $\theta>0$ and let $X$ be a random variable satisfying the following asymptotic equivalence on moments
\begin{align*}
\|X\|_k \asymp k^\theta.
\end{align*}
Then $X$ is sub-Weibull distributed with optimal tail parameter $\theta$, in the sense that for any $\theta^\prime<\theta$, $X$ is not $\subW (\theta^\prime)$.
\end{proposition}
%
\begin{proof}
  By the upper bound of the asymptotic equivalence assumption on moments, $X$ satisfies Property~\ref{item:2} of Theorem~\ref{th:subWeibull}, so $X \sim \subW(\theta)$. Let $\theta^\prime<\theta$. By the lower bound of the asymptotic equivalence assumption on moments, there does not exist any constant $K_2$ such that $\|X\|_k \leq K_2 k^{\theta'}$ for any $k\in\mathbb{R}$, so $X$ is not sub-Weibull with tail parameter $\theta^\prime$. This concludes the proof.
\end{proof}
When sub-Gaussian and sub-Exponential properties are studied, it is often assumed that the distribution is centered. However, note that these properties, as well as the sub-Weibull property, are tail properties, and as such, they are not affected by translation. Non-centered random variables can always be centered by subtracting the mean without affecting their tail property. In particular, variable centering does not change the optimal tail parameter of a sub-Weibull distribution.  Scaling procedure does not influence the tail neither: multiplication by some fixed number will change the coefficient $K_1$ in the tail Property~\ref{item:1} of Theorem~\ref{th:subWeibull},  but not the tail parameter.

It is known that the product of sub-Gaussian random variables is sub-exponential (cf. \citealt{vershynin2018high}, Lemma 2.7.7). By the same virtue, we have the fact that the class of sub-Weibull random variables is closed under multiplication and addition.
\begin{proposition}[Closure under multiplication and addition]
\label{proposition:closure}
Let $X$ and $Y$ be sub-Weibull random variables with tail parameters
$\theta_{1}$ and $\theta_{2}$, respectively. Then, $XY$ and $X+Y$ are sub-Weibull
with respective tail parameters $\theta_{1}+\theta_{2}$
and $\max(\theta_{1},\theta_{2})$.
\end{proposition}
\begin{proof}
Let us focus on the closure under addition property.
Property~\ref{item:2} in Theorem~\ref{th:subWeibull} and the triangular inequality entail that
there exist $K_{2,1}>0$ and $K_{2,2}>0$ such that,
for all $k\geq 1$,
\begin{align*}
\|X+Y\|_k\leq \|X\|_k + \|Y\|_k \leq K_{2,1} k^{\theta_1} 
+ K_{2,2} k^{\theta_2} \leq (K_{2,1}+K_{2,2}) k^{\max(\theta_1,\theta_2)}.
\end{align*}
The conclusion follows from Theorem~\ref{th:subWeibull}.
\end{proof}
Furthermore, we may establish that the class of sub-Weibull random variables is larger than the class of bounded random variables but is smaller than the class of random variables with finite $p$-th moment, for every $p\in[0,\infty)$. This result is comparable to Remark 2.7.14 of \cite{vershynin2018high}, which establishes the same relationship regarding the class of sub-Gaussian random variables.

\section{Concentration properties}\label{sec:concentration-prop}

In this section, we focus on theoretical results about the sum of sub-Weibull random variables, including concentration properties.
\begin{proposition}[Sum of sub-Weibull random variables]
\label{proposition:sum}
Let $X_1,\ldots,X_n$ be sub-Weibull random variables with tail parameter $\theta$. Then the sum  $\sum_{i=1}^{n}X_{i}$ is sub-Weibull with tail parameter $\theta$.
\end{proposition}
\begin{proof}
    Using  Property~\ref{item:2} in Theorem~\ref{th:subWeibull} combined with Minkowski's inequality yields (for any $k\geq 1$):
         \begin{align}
         \label{eq:proof-prop}
      \left\| \sum_{i=1}^{n}X_{i} \right\|_k  
      &\le  \sum_{i=1}^{n} \|X_{i}\|_k  
      \le \sum_{i=1}^{n} K_{2,i} k^{\theta} = K_{2,\Sigma} k^{\theta},
    \end{align}
where $K_{2,\Sigma} = \sum_{i=1}^{n} K_{2,i}$. 
It implies that the sum $X_1 + \cdots + X_n$ satisfies the sub-Weibull property with tail parameter at most $\theta$. In addition, if $X_1, \ldots, X_n$ are from the same distribution, i.e. $\|X_i\|_k \le K_2 k^{\theta}$ for all $i \in \{1, \dots, n\}$ with some constant $K_2 > 0$, then $K_{2,\Sigma} = n K_2$. 
\end{proof}
Using Proposition~\ref{proposition:closure}, and the proof of Property~\ref{item:2} in Theorem~\ref{th:subWeibull},
we have the following Hoeffding-type concentration inequality regarding the
sum $\sum_{i=1}^{n}X_{i}$.
\begin{corollary}[Concentration of the sum]
\label{corollary: sum concentration}
Let that $X_{1},\dots,X_{n}$ be identically distributed sub-Weibull random variables with tail parameter $\theta$. Then, for all $x \ge nK_\theta$, we have
\begin{align*}
\mathbb{P}\left(\left|\sum_{i=1}^{n}X_{i}\right|\ge x\right)\le\exp\left(-\left(\frac{x}{nK_{\theta}}\right)^{1/\theta}\right)\text{,}
\end{align*}
for some constant $K_{\theta}$ dependent on $\theta$.
\end{corollary}


\begin{proof}
For all real $k\geq 1$, the bound~(\ref{eq:proof-prop}) in
   the proof of  Proposition~\ref{proposition:sum} yields
\begin{align*}
   \mathbb{E}\left(\left|\sum_{i=1}^{n}X_{i}\right|^{k}\right)\le\left[K_{2,\Sigma}k^{\theta}\right]^{k} = \left[nK_{2}k^{\theta}\right]^{k},
\end{align*}
 and thus, by Markov's inequality,
\begin{align*}
    \mathbb{P}\left(\left|\sum_{i=1}^{n}X_{i}\right|\ge t\right) =\mathbb{P}\left(\left|\sum_{i=1}^{n}X_{i}\right|^{k}\ge t^{k}\right)
     \le\frac{\mathbb{E}\left|\sum_{i=1}^{n}X_{i}\right|^{k}}{t^{k}}\le\frac{\left(nK_{2}\right)^{k}\left(k^{\theta}\right)^{k}}{t^{k}}\text{.}
\end{align*}
Choose $t$ so that $\exp\left(-k\right)=\left(nK_{2}\right)^{k}\left(k^{\theta}\right)^{k}/t^{k}$.
Then, we have
\begin{align}
\label{eq:proof-coro}
    \mathbb{P}\left(\left|\sum_{i=1}^{n}X_{i}\right|\ge enK_{2}k^{\theta}\right)\le\exp\left(-k\right)\text{.}
\end{align}
Let $K_{\theta}=eK_{2}$ and notice that $
    nK_{\theta}k^{\theta}=x$ implies that  $k=\left(\frac{x}{nK_{\theta}}\right)^{1/\theta} $
and thus, for $x\ge nK_{\theta}$ (since $k\ge1$):
\begin{align*}
    \mathbb{P}\left(\left|\sum_{i=1}^{n}X_{i}\right|\ge x\right)\le\exp\left(-\left(\frac{x}{nK_{\theta}}\right)^{1/\theta}\right)\text{,}
\end{align*}
which is the expected result.
\end{proof}
Let us remark that, alternatively, letting $\exp\left(-k\right)=\alpha$ in~(\ref{eq:proof-coro}) implies, for $0<\alpha<1/e$, the confidence
statement
\begin{align*}
    \mathbb{P}\left(\left|\sum_{i=1}^{n}X_{i}\right|\le nK_{\theta}\left(\log\frac{1}{\alpha}\right)^{\theta}\right)\ge1-\alpha\text{.}
\end{align*}
Under the same conditions as Proposition \ref{corollary: sum concentration},
under the restriction that $\theta\le1$, \cite{boucheron2013concentration} proposed the inequality
\begin{align}
\mathbb{P}\left(\left|\sum_{i=1}^{n}X_{i}\right|\ge x\right)\le K_{\theta}\exp\left(-\frac{1}{K_{\theta}}\min\left\{ \frac{x^{2}}{n},\frac{x^{1/\theta}}{n^{\left(1-\theta\right)/\theta}}\right\} \right)\text{,}\label{eq: Boucheron concentration}
\end{align}
where $K_{\theta}$ is again a constant that only depends on $\theta$.
We observe that $\left(1-\theta\right)/\theta<1/\theta$ and thus the
bound of (\ref{eq: Boucheron concentration}) converges faster to
zero than that of Corollary \ref{corollary: sum concentration} in all cases where they are comparable. A version of this result also appears in Lemma~3.5 of \cite{adamczak2011restricted}. Further results regarding the concentration of the
sum of sub-Weibull random variables are available in \cite{kuchibhotla2018moving} and \cite{Hao:2019aa}. 
See also \cite{bakhshizadeh2020sharp} for further concentration results regarding heavy-tailed distributions.

\section{Application to Bayesian neural networks}
\label{sec:experiments}

This section gives an example of sub-Weibull variables that arise in deep learning, more specifically in the context of Bayesian deep neural networks, as originally described in \cite{vladimirova2018bayesian}. 

\subsection{Distribution at the level of units}

We first describe so-called fully connected neural networks. Neural networks are hierarchical models made of layers: an input, several hidden layers and an output. Each layer following the input layer consists of units which are linear combinations of previous layer units transformed by a nonlinear function, often referred to as the nonlinearity or activation function denoted by $\phi: \mathbb{R} \to \mathbb{R}$. Given an input $\bx \in \mathbb{R}^N$ (for instance an image made of $N$ pixels) the $\ell$-th hidden layer consists of two vectors whose size is called the width of layer, denoted by $H_\ell$. The vector of units before application of the non-linearity is called  pre-nonlinearity, and is denoted by $\bg^{(\ell)}=\bg^{(\ell)}(\bx)$, while the vector obtained after element-wise application of $\phi$  is called post-nonlinearity and is denoted by  $\bh^{(\ell)}=\bh^{(\ell)}(\bx)$. More specifically, these vectors   are defined as
\begin{align}\label{eq:propagation}
      \bg^{(\ell)}(\bx) = \bW^{(\ell)} \bh^{(\ell - 1)} (\bx), \quad \bh^{(\ell)} (\bx) = \phi(\bg^{(\ell)}(\bx)),
\end{align}
where $\bW^{(\ell)}$ is a weight matrix of dimension $H_\ell\times H_{\ell-1}$ including a bias vector, with the convention that $H_0=N$, the input dimension. 

Neural networks perform state-of-the-art results in many areas. Researchers aim at better understanding the driving mechanisms behind their effectiveness. In particular, the study of the neural networks distributional properties through Bayesian analysis, where weights are assumed to follow a prior distribution, has attracted a lot of attention in the recent years. The main focus of research in the field shows that a \textit{Bayesian deep neural network converges in distribution to a Gaussian process when the width of all the layers goes to infinity}. See for instance \cite{matthewsgaussian,lee2018deep} for the main proofs. Further research such as \cite{hayou2019impact} builds upon the limiting Gaussian process property of neural networks, as well as on the notion of \textit{edge of chaos} developed by \cite{schoenholz2016deep}, in order to devise novel architecture rules for neural networks.

\begin{figure}
  \centering
  \includegraphics[width=.5\textwidth]{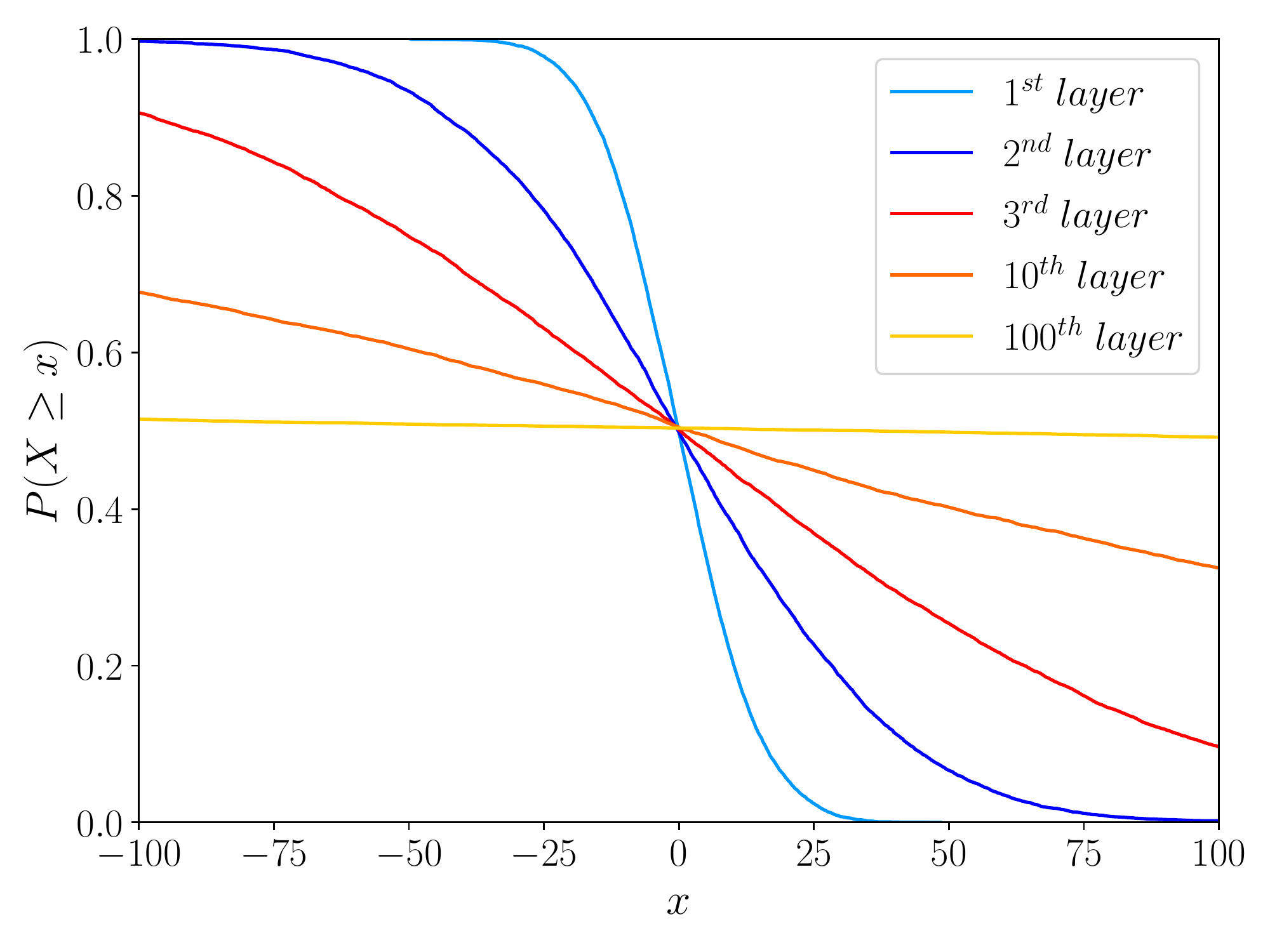}
  \caption{Illustration of layers $\ell=1,2,3,10$ and $100$ hidden units marginal prior distributions tails. According to Theorem 2.1 of \citet{vladimirova2018bayesian}, they correspond respectively to $\subW(\sfrac12)$,  $\subW(1)$,  $\subW(\sfrac32)$,  $\subW(5)$ and  $\subW(50)$.}
  \label{pic:nn_layers}
\end{figure}

In contrast, Theorem 2.1 of  \cite{vladimirova2018bayesian} shows that the \textit{non asymptotic} (i.e. for finite width neural networks) prior distribution of units from the $\ell$-th layer (both before and after activation, $\bg^{(\ell)}$ and~$\bh^{(\ell)}$) induced by a standard Gaussian prior on the weights $\bW^{(\ell)}$, is \textit{sub-Weibull with tail  parameter $\ell/2$}. Therefore, the deeper the layer is, the heavier-tailed is the units distribution. This result puts into perspective the infinite width Gaussian process property which might be far from holding for real world neural networks.

In order to illustrate the sub-Weibull property of units prior distributions, we performed the following experiment with a deep neural network of 100 layers. 
We considered a Bayesian neural network with independent standard Gaussian priors on the weights, where layers are composed of
$H_\ell = 1000 - 10(\ell-1)$ hidden units, $\ell=1,\ldots,100$, with the ReLU nonlinearity $\phi$ defined by $\phi(x) = \max(0,x)$ (see \eqref{eq:propagation}). The input vector $\bx$ contains $10^4$ numbers sampled from independent standard Gaussian. 
In order to evaluate the units prior distributions, we used a Monte Carlo approximation, where the input vector $\bx$ was kept fixed, and the weights were sampled from independent standard Gaussian $n=10^5$ times. 
Figure~\ref{pic:nn_layers} illustrates the survival function of pre-nonlinearity hidden units $\bg^{(\ell)}$ for layers $\ell=1, 2, 3, 10$ and 100. 
This indicates that the prior tails of the units get heavier when they originate from deeper layers, in accordance with the main result of \citet{vladimirova2018bayesian}.

\subsection{Tail parameter estimation}

We conclude this section by suggesting a statistical estimation procedure for the tail parameter $\theta$. 
We adopt an approach that relies on Weibull random variables with tail parameter $\theta>0$ and scale parameter  $\lambda>0$, which are $\subW(\theta)$. 
The cumulative distribution function is $F(y) = 1-\edr^{-(y/\lambda)^{1/\theta}}$ for all $y \geq 0$, with corresponding quantile function $q(t) = \lambda \left(-\log(1-t)\right)^\theta$. Taking the logarithm, we obtain the following expression
\begin{align*}
    \log q(t) = \theta\log\log\frac{1}{1-t} + \log\lambda,
\end{align*}
showing an affine relationship involving the log quantiles and $\theta$ as slope parameter. This suggests a simple estimation procedure for $\theta$ by linear regression. More specifically, assume we have $n$ iid observations  $Y_1, \ldots, Y_n$, and denote the order statistics by  $Y_{i,n}$. 
Consider a fraction $k\leq n$ of the largest observations, that are the order statistics  $Y_{n-i+1,n}$ for $1\leq i\leq k$. 
These are approximating the quantiles of order $\frac{n-i}{n}=1-\frac{i}{n}$, respectively. As a consequence, we consider the slope coefficient in the regression of $\log Y_{n-i+1,n}$ against $\log\log(n/i)$, for $1\leq i\leq k$, as an estimate of $\theta$. This estimator was first introduced by \citet{gardes2008estimation} in the context of Weibull-tail distributions.
We visualize the alignment in a quantile-quantile plot in log-log scale in Figure~\ref{fig:BNN}.

We considered the following experiment in order to illustrate the estimation procedure. 
We have built neural networks of 10 hidden layers, with $H_\ell=H$ hidden units on each layers, and made $H$ vary in $\{1,3,10\}$. 
We used a fixed input $\bx$ of size $10^4$, which can be thought of as an image of dimension $100\times 100$. This input was sampled once for all with standard Gaussian entries. 
In order to obtain samples from the prior distribution of the neural network units, we have sampled the weights from independent centered Gaussians with variance set to 2, from which units were obtained by forward evaluation with the ReLU non-linearity. This process was iterated $n =10^5$ times and  the  number of largest observations selected is $k = 10^3$.
We show in Figure~\ref{fig:BNN} the $\theta$ estimates obtained by using pre-nonlinearities $\bg^{(\ell)}$ of layers $\ell$ in $\{1,3,10\}$ as data, for $H$ varying in $\{1,3,10\}$. 
We can see that the theoretical result from \cite{vladimirova2018bayesian} that states that $\theta = \ell/2$ is well in line with the estimates obtained with networks of width $H=1$. When the width increases, the estimates for $\theta$ tend to decrease, narrowing the gap to the lower bound of $1/2$. These results illustrate the large width results by \cite{matthewsgaussian,lee2018deep}, which state that in the large width regime, the unit distribution tends to a Gaussian, thus with a limiting $\theta$ parameter of $1/2$.
\begin{figure}
  \centering
  \includegraphics[width=.25\textwidth]{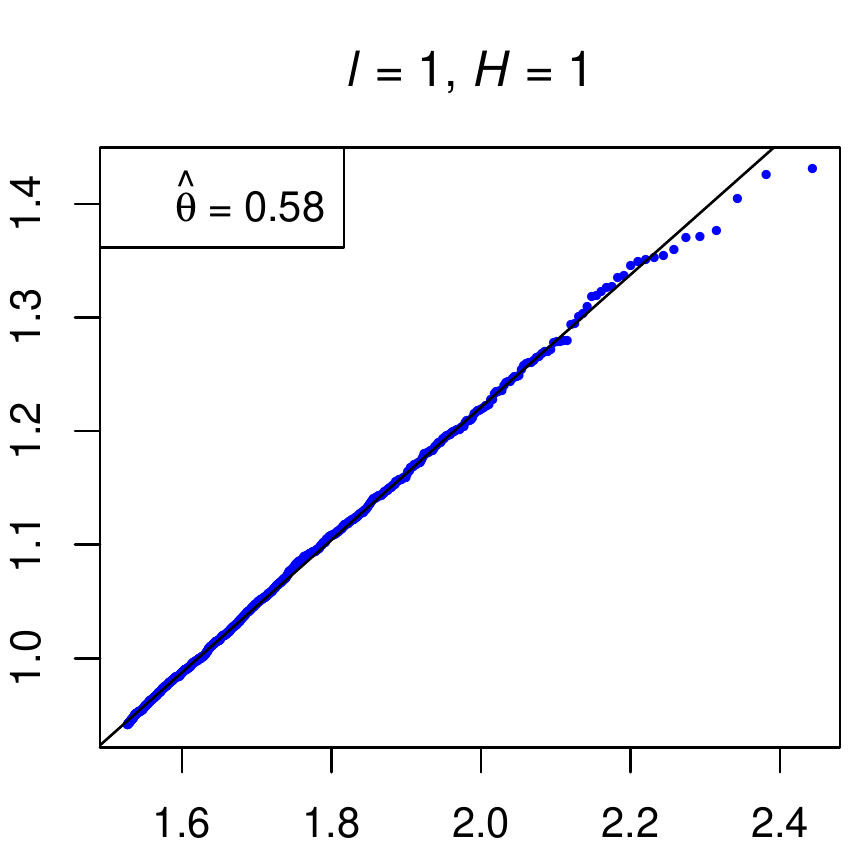}
  \includegraphics[width=.25\textwidth]{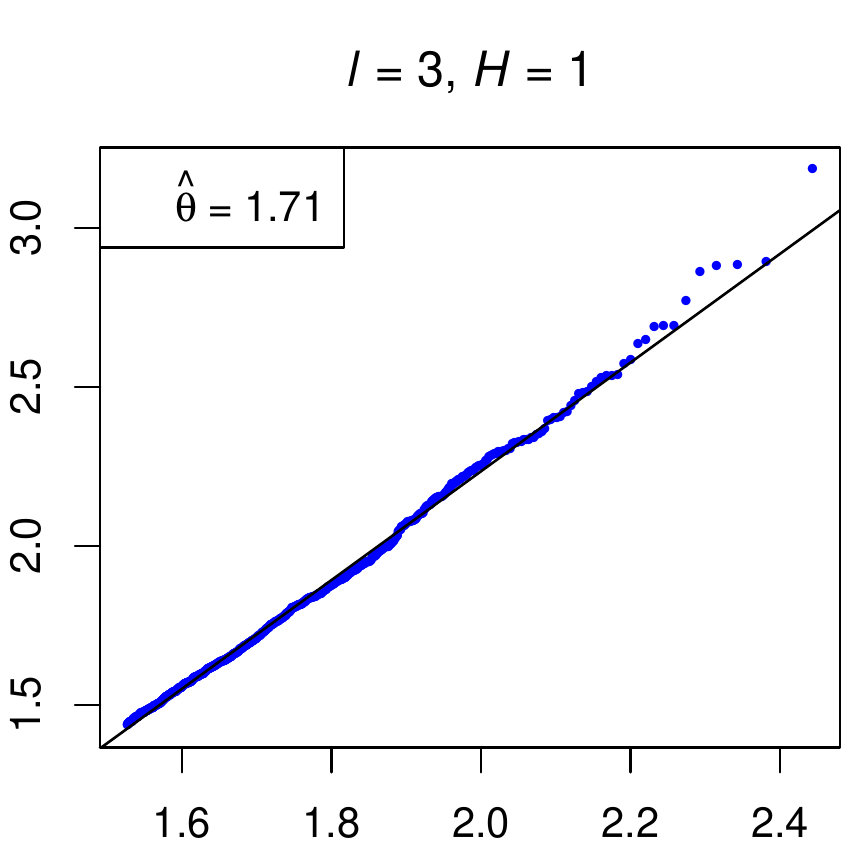}
  \includegraphics[width=.25\textwidth]{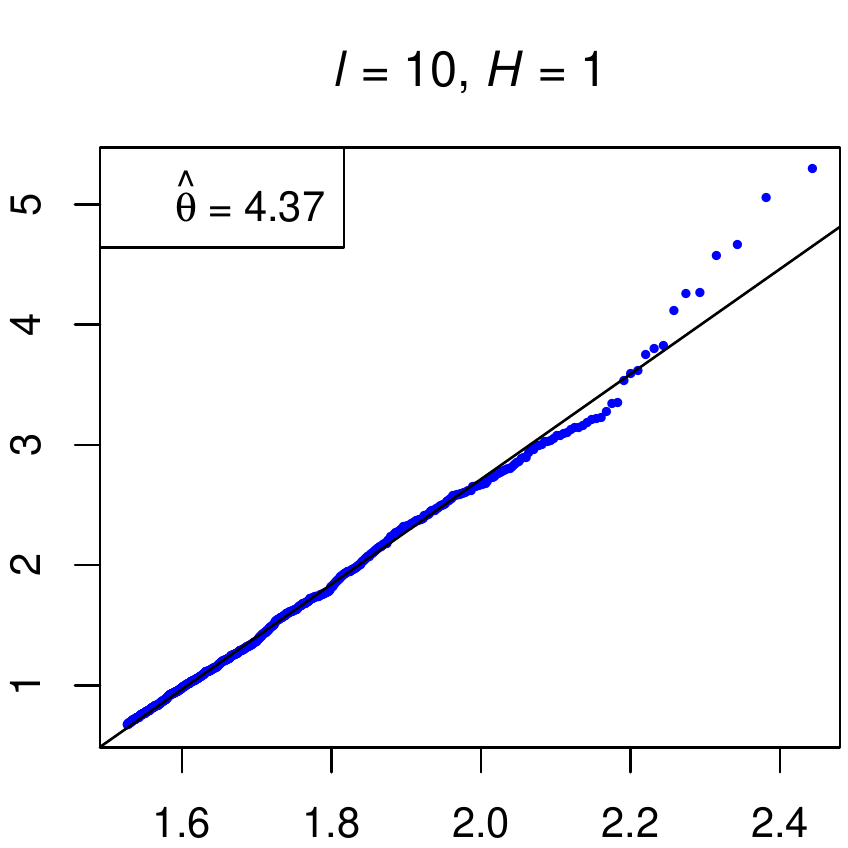}
  \includegraphics[width=.25\textwidth]{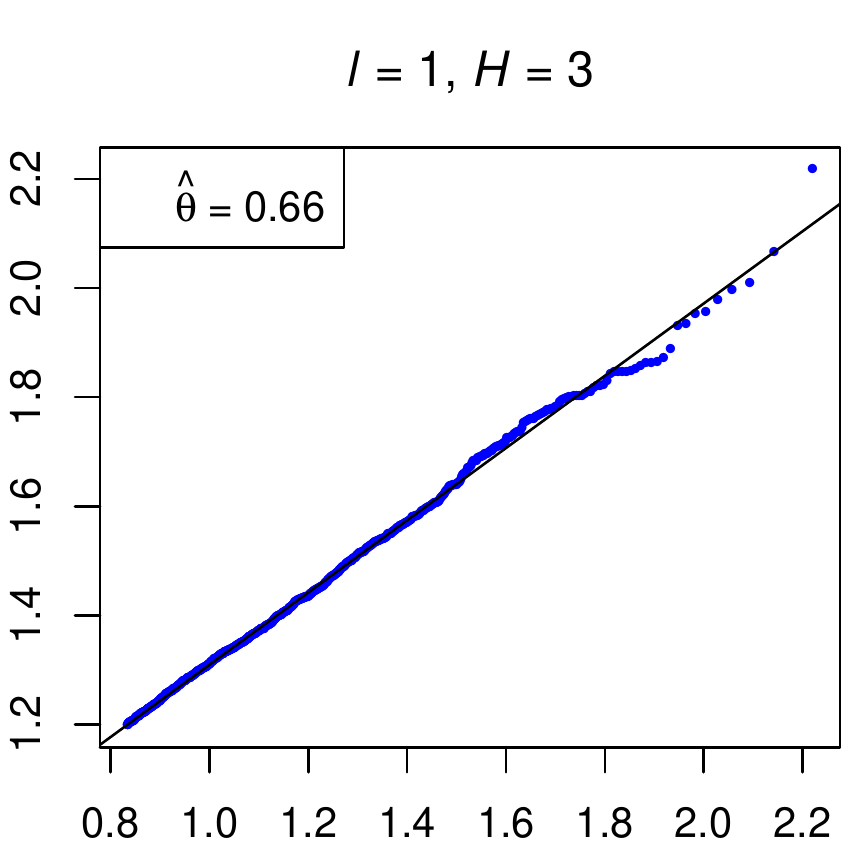}
  \includegraphics[width=.25\textwidth]{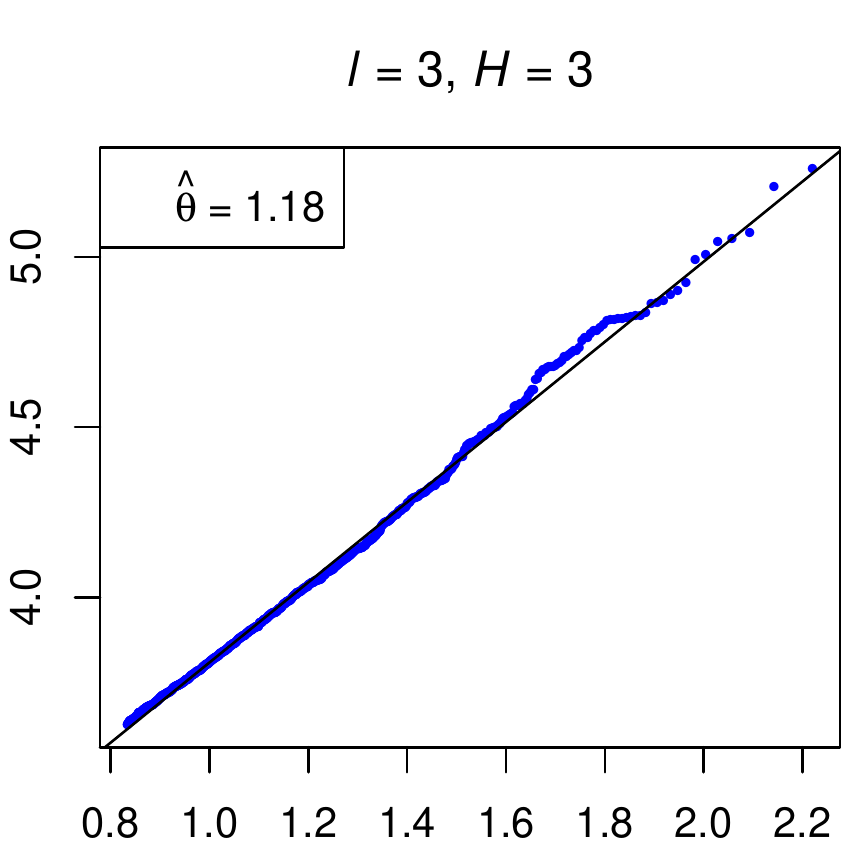}
  \includegraphics[width=.25\textwidth]{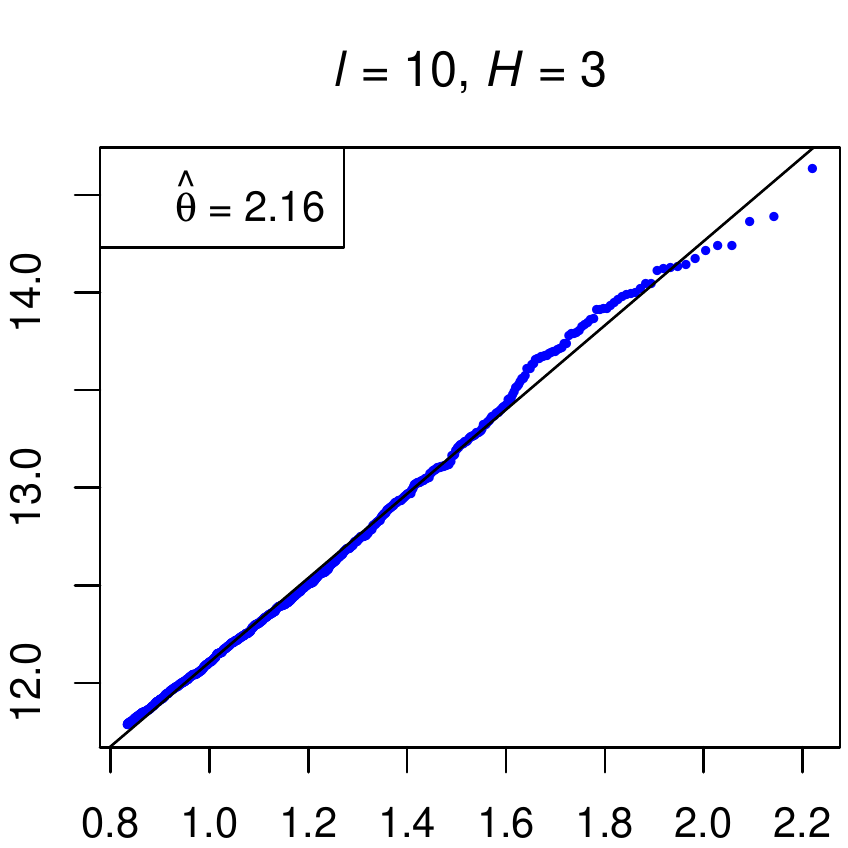}
  \includegraphics[width=.25\textwidth]{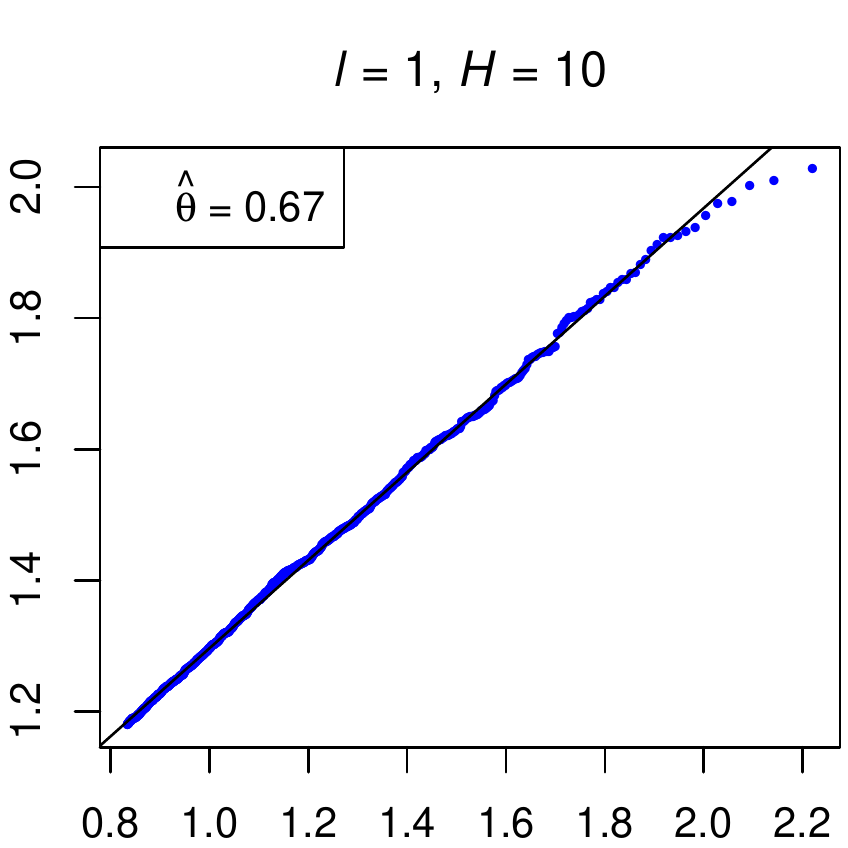}
  \includegraphics[width=.25\textwidth]{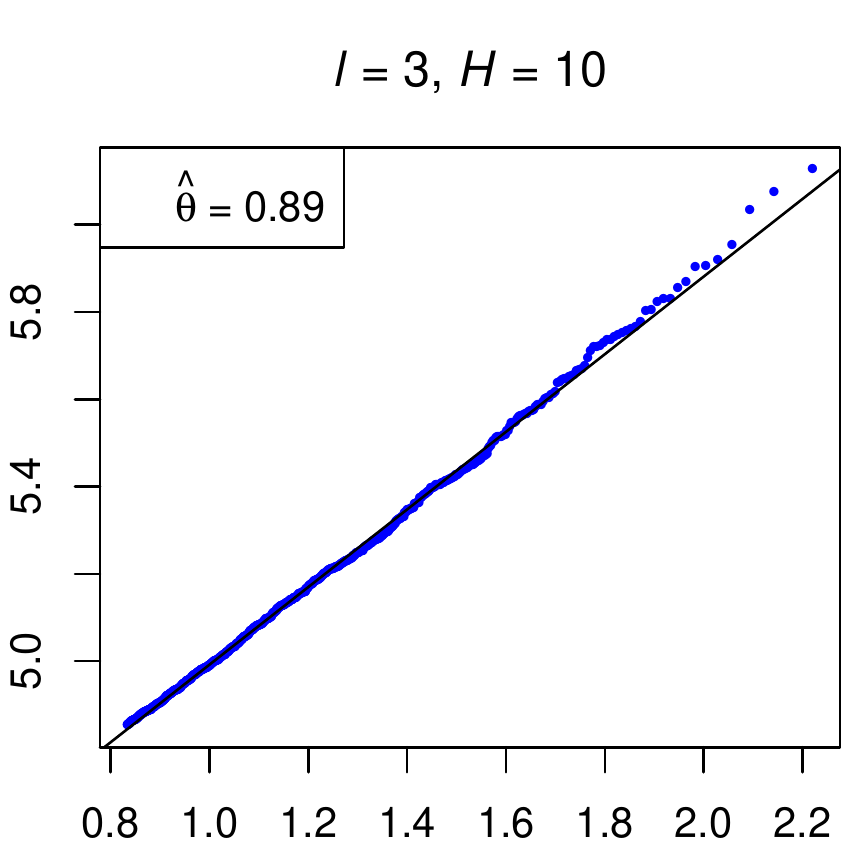}
  \includegraphics[width=.25\textwidth]{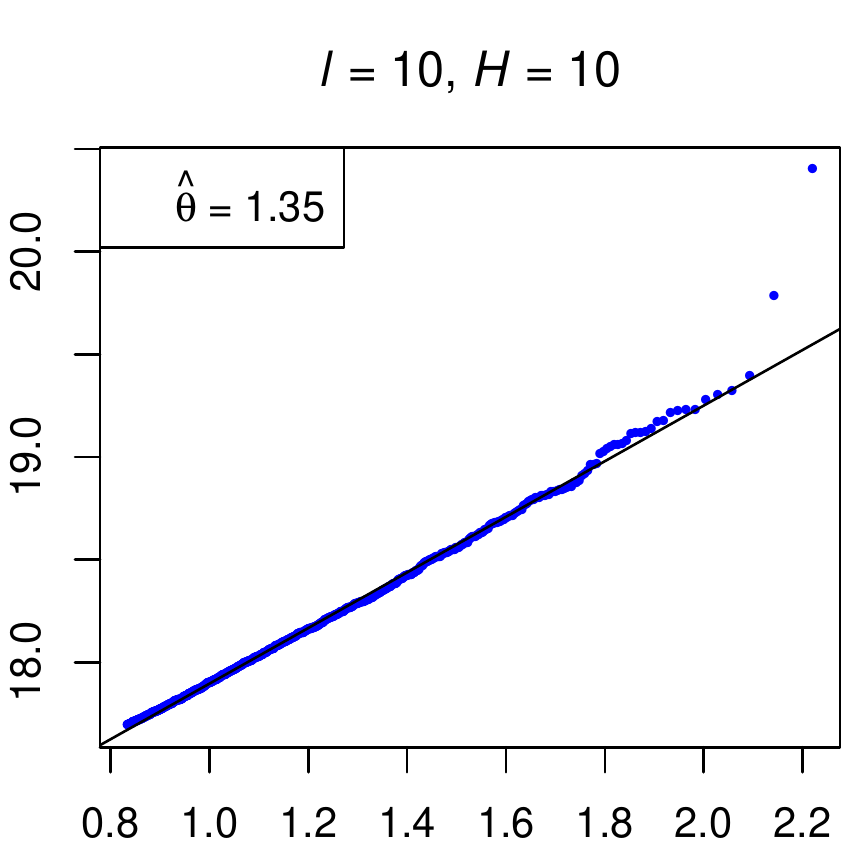}
  \caption{Illustration of the estimation procedure for the tail parameter $\theta$ for the prior distribution of Bayesian neural networks units. 
  Layers  $\ell$  of level 1, 3 and 10 on the left, middle and right panels, respectively. Width $H$ (for all considered layers  $\ell$ ) equal to 1, 3 and 10 on the top, middle and bottom rows, respectively. Note that the theoretical tail parameter value stated by \cite{vladimirova2018bayesian} is $\theta = \ell/2$, while the infinite width ($H\to\infty$) limiting value is $\theta = 1/2$.}
  \label{fig:BNN}
\end{figure}

\clearpage

\bibliographystyle{apalike}

\begin{thebibliography}{}

\bibitem[Adamczak et~al., 2011]{adamczak2011restricted}
Adamczak, R., Litvak, A.~E., Pajor, A., and Tomczak-Jaegermann, N. (2011).
\newblock Restricted isometry property of matrices with independent columns and
  neighborly polytopes by random sampling.
\newblock {\em Constructive Approximation}, 34:61–88.

\bibitem[Bakhshizadeh et~al., 2020]{bakhshizadeh2020sharp}
Bakhshizadeh, M., Maleki, A., and H.~de~la Pena, V. (2020).
\newblock Sharp concentration results for heavy-tailed distributions.
\newblock {\em arXiv preprint arXiv:2003.13819}.

\bibitem[Ben-Hamou et~al., 2017]{ben2017concentration}
Ben-Hamou, A., Boucheron, S., and Ohannessian, M.~I. (2017).
\newblock Concentration inequalities in the infinite urn scheme for occupancy
  counts and the missing mass, with applications.
\newblock {\em Bernoulli}, 23:249--287.

\bibitem[Boucheron et~al., 2013]{boucheron2013concentration}
Boucheron, S., Lugosi, G., and Massart, P. (2013).
\newblock {\em {Concentration Inequalities: A Nonasymptotic Theory of
  Independence}}.
\newblock Oxford University Press, Oxford.

\bibitem[Bubeck and Cesa-Bianchi, 2012]{bubeck2012regret}
Bubeck, S. and Cesa-Bianchi, N. (2012).
\newblock {\em {Regret Analysis of Stochastic and Nonstochastic Multi-Armed
  Bandit Problems}}.
\newblock Foundations and Trends{\textregistered} in Machine Learning. Now
  Publishers, Delft.

\bibitem[Catoni, 2007]{catoni2007pac}
Catoni, O. (2007).
\newblock {\em {PAC-Bayesian Supervised Classification: The Thermodynamics of
  Statistical Learning}}.
\newblock Monograph Series. Institute of Mathematical Statistics, Lithuania.

\bibitem[Gardes and Girard, 2008]{gardes2008estimation}
Gardes, L. and Girard, S. (2008).
\newblock {Estimation of the Weibull tail-coefficient with linear combination
  of upper order statistics}.
\newblock {\em {Journal of Statistical Planning and Inference}},
  138:1416--1427.

\bibitem[Hao et~al., 2019]{Hao:2019aa}
Hao, B., Yadkori, Y.~A., Wen, Z., and Cheng, G. (2019).
\newblock Bootstrapping upper confidence bound.
\newblock In {\em {Advances in Neural Information Processing Systems}}.

\bibitem[Hayou et~al., 2019]{hayou2019impact}
Hayou, S., Doucet, A., and Rousseau, J. (2019).
\newblock On the impact of the activation function on deep neural networks
  training.
\newblock In {\em {International Conference on Machine Learning}}.

\bibitem[Kuchibhotla and Chakrabortty, 2018]{kuchibhotla2018moving}
Kuchibhotla, A.~K. and Chakrabortty, A. (2018).
\newblock {Moving beyond sub-Gaussianity in high-dimensional statistics:
  Applications in covariance estimation and linear regression}.
\newblock {\em arXiv preprint arXiv:1804.02605}.

\bibitem[Lee et~al., 2018]{lee2018deep}
Lee, J., Sohl-Dickstein, J., Pennington, J., Novak, R., Schoenholz, S., and
  Bahri, Y. (2018).
\newblock Deep neural networks as {G}aussian processes.
\newblock In {\em {International Conference on Learning Representations}}.

\bibitem[Matthews et~al., 2018]{matthewsgaussian}
Matthews, A. G. d.~G., Rowland, M., Hron, J., Turner, R.~E., and Ghahramani, Z.
  (2018).
\newblock Gaussian process behaviour in wide deep neural networks.
\newblock In {\em {International Conference on Learning Representations}}.

\bibitem[Raginsky and Sason, 2013]{raginsky2013concentration}
Raginsky, M. and Sason, I. (2013).
\newblock Concentration of measure inequalities in information theory,
  communications, and coding.
\newblock {\em Foundations and Trends in Communications and Information
  Theory}, 10:1--246.

\bibitem[Rinne, 2008]{rinne2008weibull}
Rinne, H. (2008).
\newblock {\em {The Weibull Distribution: a Handbook}}.
\newblock CRC Press, Boca Raton.

\bibitem[Rudelson and Vershynin, 2010]{rudelson2010non}
Rudelson, M. and Vershynin, R. (2010).
\newblock Non-asymptotic theory of random matrices: extreme singular values.
\newblock In {\em {International Congress of Mathematicians}}. World
  Scientific.

\bibitem[Schoenholz et~al., 2017]{schoenholz2016deep}
Schoenholz, S., Gilmer, J., Ganguli, S., and Sohl-Dickstein, J. (2017).
\newblock Deep information propagation.
\newblock In {\em {International Conference on Learning Representations}}.

\bibitem[van Handel, 2014]{van2014probability}
van Handel, R. (2014).
\newblock Probability in high dimension.
\newblock Technical report, Princeton University, Princeton.

\bibitem[Vershynin, 2018]{vershynin2018high}
Vershynin, R. (2018).
\newblock {\em {High-Dimensional Probability: An Introduction with Applications
  in Data Science}}.
\newblock Cambridge University Press, Cambridge.

\bibitem[Vladimirova et~al., 2019]{vladimirova2018bayesian}
Vladimirova, M., Verbeek, J., Mesejo, P., and Arbel, J. (2019).
\newblock Understanding priors in {B}ayesian neural networks at the unit level.
\newblock In {\em {International Conference on Machine Learning}}.

\bibitem[Wellner, 2017]{wellner2017bennett}
Wellner, J.~A. (2017).
\newblock {The Bennett-Orlicz Norm}.
\newblock {\em {Sankhya A}}, 79:355--383.

\end{thebibliography}

\end{document}